\documentclass[10pt,a4paper]{amsart}
\usepackage{amsmath,amssymb,latexsym,tabularx}
\usepackage[dvips]{graphicx}
\headsep=1.2500cm \numberwithin{equation}{section}
\newtheorem{theorem}{Theorem}[section]
\newtheorem{proposition}[theorem]{Proposition}
\newtheorem{corollary}[theorem]{Corollary}
\newtheorem{lemma}[theorem]{Lemma}
\theoremstyle{definition}
\newtheorem{example}[theorem]{Example}
\newtheorem{definition}[theorem]{Definition}
\theoremstyle{remark}

\title[  On Large Scale Inductive Dimension of Asymptotic Resemblance Spaces]{ On Large Scale Inductive Dimension of Asymptotic Resemblance Spaces}

\author[Sh. Kalantari]{Sh. Kalantari}
 \address[Sh. Kalantari]{Mathematics and Computer Science Department,
 Amirkabir University of Technology, 424 Hafez Avenue, 15914 Tehran,
Iran.}
 \email{shahab.kalantari@aut.ac.ir}

\author[B.~Honari]{B. Honari$^\dag$}

\address[B.~Honari]{Mathematics and Computer Science Department,
 Amirkabir University of Technology, 424 Hafez Avenue, 15914 Tehran,
Iran.}

\address[B.~Honari]{$^\dag$ Corresponding author.}
\email{honari@aut.ac.ir}

%Extension of maps; metric geometry (others); Categories of topological space and continuous mappings; Global geometric and topological methods (a la Gromov), differential geometric analysis on metric spaces

\begin{document}
\maketitle
\begin{abstract}
We introduce the notion of \emph{large scale inductive dimension} for \emph{asymptotic resemblance spaces}. We prove that the large scale inductive dimension and the \emph{asymptotic dimensiongrad} are equal in the class of \emph{$r$-convex} metric spaces. This class contains the class of all geodesic metric spaces and all finitely generated groups. This leads to an answer to a question asked by E. Shchepin, concerning the relation between the \emph{asymptotic inductive dimension} and the asymptotic dimensiongrad, for $r$-convex metric spaces.
\end{abstract}
\section{introduction}
\emph{Asymptotic dimension} was first introduced by Gromov as a large scale invariant of metric spaces \cite{gro}. The asymptotic dimension has found many applications in the area of geometric group theory. In \cite{Ind} and \cite{Uni}, Dranishnikov and Zarichnyi gave a definition for \emph{asymptotic inductive dimension} of a proper metric space. To find a definition for a separator in large scale, they considered subsets of a proper metric space which can separate the boundary of two asymptotically disjoint subsets in the Higson  corona (section \ref{111}). For a topological space with a compatible and proper coarse structure this definition can be used without making any changes (Section \ref{111}).\\
 In \cite{maa}, motivated by the definition of \emph{proximity space} (\cite{Ef2}) in small scale, we introduced the notion of \emph{asymptotic resemblance space} as a large scale structure on a set. An asymptotic resemblance relation on a set $X$, is an equivalence relation on the family of all subsets of $X$ with two additional properties (Definition \ref{asr}). In section 3, using an inductive approach for finding the dimension of an asymptotic resemblance space, we define the \emph{large scale inductive dimension}. The large scale inductive dimension can be considered as a generalization of asymptotic inductive dimension for asymptotic resemblance spaces. The large scale inductive dimension is an asymptotic invariant of asymptotic resemblance spaces. Since a coarse structure can induce an asymptotic resemblance relation, our definition can be used for all coarse spaces. One of the advantages of our definition is that, for finding the large scale inductive dimension of a coarse space $X$ the set $X$ does not need to have a topology. In \cite{Uni} there is also the definition of \emph{asymptotic dimensiongrad}. In section 4, we prove that the large scale inductive dimension and the asymptoic dimensiongrad are equal for the class of all metric spaces with a property that we call \emph{$r$-convexity}. The class of all $r$-convex metric spaces contains the family of all geodesic metric spaces and all finitely generated groups as two important examples. This leads to a partial answer for an open question concerning the relation between the asymptotic inductive dimension and the asymptotic dimensiongrad of a metric space, asked in \cite{Uni}.
\section{preliminaries}
\subsection{Asymptotic resemblance}
We recall here some notions from \cite{maa}.
\begin{definition}\label{asr}
An \emph{asymptotic resemblance} (an AS.R.) $\lambda$ on a set $X$ is an equivalence relation on the family of all subsets of $X$ that has the following properties,\\
i) If $A_{1}\lambda B_{1}$ and $A_{2}\lambda B_{2}$ then $(A_{1}\bigcup A_{2})\lambda (B_{1}\bigcup B_{2})$.\\
ii) If $(B_{1}\bigcup B_{2})\lambda A$ and $B_{1},B_{2}\neq \emptyset$ then there are nonempty subsets $A_{1}$ and $A_{2}$ of $A$ such that $A=A_{1}\bigcup A_{2}$ and $B_{i}\lambda A_{i}$ for $i\in \{1,2\}$.\\
We call the pair $(X,\lambda)$ an AS.R. space. We call two subsets $A$ and $B$ of $X$ asymptotically alike if $A\lambda B$. By $A\bar{\lambda} B$, we mean $A$ and $B$ are not asymptotically alike.
\end{definition}
\begin{proposition} (Proposition 2.6 of \cite{maa})\label{6setare}
Let $(X,\lambda)$ be an AS.R. space. If $A\lambda B$ and $A_{1}$ is a nonempty subset of $A$ then there is a nonmepty subset $B_{1}$ of $B$ such that $A_{1}\lambda B_{1}$.
\end{proposition}
\begin{example}
Let $(X,d)$ be a metric space. For two subsets $A$ and $B$ of $X$, define $A\lambda_{d}B$ if $d_{H}(A,B)<\infty$. By $d_{H}(A,B)$ we mean the Hausdorff distance between $A$ and $B$. The relation $\lambda_{d}$ is an AS.R on $X$, which we call the asymptotic resemblance associated to the metric $d$.
\end{example}
Let us recall that a coarse structure $\mathcal{E}$ on a set $X$, is a family of subsets of $X\times X$ such that $\mathcal{E}$ contains any subset of its members and $E\bigcup F$, $E\circ F$ and $E^{-1}$ are in $\mathcal{E}$, for each $E,F\in \mathcal{E}$. For two subset $E$ and $F$ of $X\times X$, $E\circ F=\{(x,y)\mid (x,z)\in F,(z,y)\in E\,for\,some\,z\in X\}$ and $E^{-1}=\{(x,y)\mid (y,x)\in E\}$.
\begin{example}
Let $(X,\mathcal{E})$ be a coarse space. For two subsets $A$ and $B$ of $X$, define $A\lambda_{\mathcal{E}}B$ if there is an $E\in \mathcal{E}$ such that $A\subseteq E(B)$ and $B\subseteq E(A)$. The relation $\lambda_{\mathcal{E}}$ is an AS.R. on $X$. This is called the asymptotic resemblance associated to the coarse structure $\mathcal{E}$.
\end{example}
\begin{definition}
Let $\lambda$ be an AS.R. on a set $X$ and $A\subseteq X$. We say $A$ is \emph{bounded} if there exists some $x\in X$ such that $A\lambda x$. We notice that the empty set is bounded.
\end{definition}
\begin{definition}\label{map}
Let $(X,\lambda_{1})$ and $(Y,\lambda_{2})$ be two AS.R. spaces. We say a map $f:X\rightarrow Y$ is an \emph{AS.R. mapping} if\\
i) If $B\subseteq Y$ is bounded then $f^{-1}(B)$ is a bounded subset of $X$. (Properness)\\
ii) For all subsets $A$ and $B$ of $X$, If $A\lambda_{1}B$ then $f(A)\lambda_{2}f(B)$.
\end{definition}
Let $(X,d)$ and $(Y,d^{\prime})$ be two metric spaces. A map $f:X\rightarrow Y$ is a called a coarse map if for each $r>0$ there exist some $s>0$ such that $d(x,y)<r$ implies $d^{\prime}(f(x),f(y))<s$. A map $f:X\rightarrow Y$ is a coarse map if and only if it is an asymptotic mapping with respect to $\lambda_{d}$ and $\lambda_{d^{\prime}}$ (\cite{maa} Theorem 2.3).
\begin{definition}
Let $X$ and $Y$ be two sets and let $\lambda$ be an AS.R. on $Y$. We say that two maps $f:X\rightarrow Y$ and $g:X\rightarrow Y$ are \emph{close} if $f(A)\lambda g(A)$ for all $A\subseteq X$.
\end{definition}
\begin{definition}
Assume that $(X,\lambda_{1})$ and $(Y,\lambda_{2})$ are two AS.R. spaces. We call an AS.R. mapping $f:X\rightarrow Y$ an \emph{asymptotic equivalence} if there exists an AS.R. mapping $g:Y\rightarrow X$ such that $gof$ and $fog$ are close to the identity maps $i_{X}:X\rightarrow X$ and $i_{Y}:Y\rightarrow Y$ respectively. We say AS.R. spaces $(X,\lambda_{1})$ and $(Y,\lambda_{2})$ are \emph{asymptotically equivalent} if there exists an asymptotic equivalence $f:X\rightarrow Y$.
\end{definition}
Let $(X,d)$ and $(Y,d)$ be two metric spaces. It can be proved that $(X,\lambda_{d})$ and $(Y,\lambda_{d^{\prime}})$ are asymptotically equivalent if and only if $(X,d)$ and $(Y,d^{\prime})$ are coarsely equivalent (\cite{maa}).
\begin{definition}
Let $(X,\lambda)$ be an AS.R. space and let $Y$ be a nonempty subset of $X$. For two subsets $A$ and $B$ of $Y$, define $A\lambda_{Y}B$ if $A\lambda B$. The pair $(Y,\lambda_{Y})$ is an AS.R. space and we call $\lambda_{Y}$ the \emph{subspace} AS.R. induced by $\lambda$ on $Y$.
\end{definition}
\begin{proposition}\label{sub}(\cite{maa} \emph{Lemma 2.20})
Let $(X,\lambda)$ and $(Y,\lambda^{\prime})$ be two AS.R. spaces. Suppose that $f:X\rightarrow Y$ is an asymptotic equivalent and $C\subseteq X$. Then $f\mid_{C}:(C,\lambda_{C})\rightarrow (f(C),\lambda^{\prime}_{f(C)})$ is an asymptotic equivalent too.
\end{proposition}
\subsection{asymptotic compactification}
We recall here what is the asymptotic compactification of an AS.R. space. For details see \cite{maa}.
\begin{definition}
Let $(X,\lambda)$ is an AS.R. space. We call two subsets $A_{1}$ and $A_{2}$ of $X$ \emph{asymptotically disjoint} if $L_{1}\bar{\lambda} L_{2}$, for all unbounded subsets $L_{1}\subseteq A_{1}$ and $L_{2}\subseteq A_{2}$. We say that the AS.R. space $(X,\lambda)$ is \emph{asymptotically normal} if for asymptotically disjoint subsets $A_{1}$ and $A_{2}$ of $X$, there exist $X_{1}\subseteq X$ and $X_{2}\subseteq X$ such that $X=X_{1}\bigcup X_{2}$ and $A_{i}$ and $X_{i}$ are asymptotically disjoint for $i\in \{1,2\}$.
\end{definition}
Notice that if $D$ is a bounded subset of an AS.R. space $(X,\lambda)$, then it is asymptotically disjoint from each subset of $X$ (Since $D$ does not contain any unbounded subset).
Let $(X,d)$ be a metric space. It can be proved that the AS.R. space $(X,\lambda_{d})$ is asymptotically normal (\cite{maa}). Two subsets $A$ and $B$ are called asymptotically disjoint with respect to $d$ if for some $x_{0}\in X$, $\lim_{r\rightarrow \infty} d(A\setminus\textbf{B}(x_{0},r),B\setminus\textbf{B}(x_{0},r))=\infty$ (\cite{Ind} and \cite{Uni}). It can be shown that two unbounded subsets of $(X,\lambda_{d})$ are asymptotically disjoint if and only if they are asymptotically disjoint with respect to $d$.
 \begin{definition}
Let $(X,\mathcal{T})$ be a topological space. We call an AS.R. $\lambda$ on $X$ \emph{compatible} with $\mathcal{T}$ if\\
i) For each subset $A$ of $X$ there exists an open subset $A\subseteq U$ such that $A\lambda U$.\\
ii) $A \lambda \bar{A}$ for all $A\subseteq X$.\\
A compatible AS.R. on $X$ is called \emph{proper} if each bounded subset of $X$ has compact closure.
\end{definition}
Let $X$ be a topological space. A coarse structure $\mathcal{E}$ on $X$ is said to be \emph{compatible with the topology} of $X$ if for each $E\in \mathcal{E}$ there exists an open $F\in \mathcal{E}$ such that $E\subseteq F$. A compatible coarse structure on a topological space is called \emph{proper} if each bounded subset has compact closure. It can be shown that the associated AS.R. to a compatible (proper) coarse structure is a compatible (proper) AS.R.  Let $\mathcal{E}$ be a proper coarse structure on a topological space $X$. A continuous and bounded map $f:X\rightarrow \mathbb{C}$ is called a \emph{Higson function} if for each $E\in \mathcal{E}$ and $\epsilon >0$ there exists a compact subset $K$ of $X$ such that $| f(x)-f(y)| <\epsilon$ for all $(x,y)\in E\setminus (K\times K)$. The family of all Higson functions is denoted by $C_{h}(X)$. There is a compactification $hX$ of $X$, such that the family of all continuous functions on $hX$ and $C_{h}(X)$ are isomorphic (section 2.3 of \cite{Roe}). The compactification $hX$ of $X$ is called the \emph{Higson compactification} of $X$. The compact set $\nu X=hX\setminus X$ is called the \emph{Higson corona} of $X$. For a subset $A$ of $X$ we denote by $\bar{A}$ the closure of $A$ in $hX$ and by $\nu A$ the intersection of $\bar{A}$ and the Higson corona of $X$.\\
Let us recall here some notions about \emph{proximity spaces}.
\begin{definition}
A relation $\delta$ on the family of all subsets of a nonempty set $X$ is called a \emph{proximity} on $X$ if for all $A,B,C\subseteq X$,\\
i) If $A\delta B$ then $B\delta A$.\\
ii) $\emptyset \bar{\delta} A$.\\
iii) If $A\bigcap B\neq \emptyset$ then $A\delta B$.\\
iv) $A\delta(B\bigcup C)$ if and only if $A\delta B$ or $A\delta C$.\\
v) If $A\bar{\delta} B$ then there exists some $E\subseteq X$ such that $A\bar{\delta} E$ and $(X\setminus E)\bar{\delta} B$, where $A\bar{\delta} B$ means that $A\delta B$ is not true.\\
The pair $(X,\delta)$ is called a proximity space.
\end{definition}
A \emph{cluster} $\mathcal{C}$ in a proximity space $(X,\delta)$ is a family of subsets of $X$ such that, i) for all $A,B\in \mathcal{C}$ we have $A\delta B$, ii) if $A,B\subseteq X$ and $A\bigcup B\in \mathcal{C}$ then $A\in \mathcal{C}$ or $B\in \mathcal{C}$, iii) if $A\delta B$ for all $B\in \mathcal{C}$ then $A\in \mathcal{C}$. A proximity space $(X,\delta)$ is said to be \emph{separated} if $x\delta y$ implies $x=y$, for all $x,y\in X$. A proximity $\delta$ on a topological space $(X,\mathcal{T})$ is said to be \emph{compatible with $\mathcal{T}$} if $a\in \bar{A}$ and $a\delta A$ are equivalent. Let $\mathfrak{X}$ denotes the family of all clusters in a separated proximity space $(X,\delta)$. For $\mathfrak{M} , \mathfrak{N} \subseteq \mathfrak{X}$ define $\mathfrak{M} \delta^{*} \mathfrak{N}$ if $A\subseteq X$ absorbs $\mathfrak{M}$ and $B\subseteq X$ absorbs $\mathfrak{N}$ then $A\delta B$. A set $D$ \emph{absorbs} $\mathfrak{M} \subseteq \mathfrak{X}$ means that $D\in \mathcal{C}$ for all $\mathcal{C} \in \mathfrak{M}$. The relation $\delta^{*}$ is a proximity on $\mathfrak{X}$. The pair $(\mathfrak{X},\delta^{*})$ is a compact proximity space and it is called the \emph{Smirnov compactification } of $(X,\delta)$ (section 7 of \cite{Nai}).\\
Let $(X,\mathcal{T})$ be a topological space and $\lambda$ be an AS.R. compatible with $\mathcal{T}$. For two nonempty subsets $A$ and $B$ of $X$ define $A\sim B$ if $A=B$ or $A$ and $B$ are unbounded asymptotically alike subsets of $X$. For two subsets $A$ and $B$ of $X$, define $A\delta_{\lambda} B$ if there are $L_{1}\subseteq \bar{A}$ and $L_{2}\subseteq \bar{B}$ such that $L_{1}\sim L_{2}$. The following proposition is proved in \cite{maa}.
\begin{proposition}
Let $(X,\mathcal{T})$ be a normal topological space and let $\lambda$ be a proper and asymptotically normal AS.R. on $X$. Then $\delta_{\lambda}$ is a separated proximity on $X$ and it is compatible with $\mathcal{T}$.
\end{proposition}
\begin{definition}
Let $(X,\mathcal{T})$ be a normal topological space and let $\lambda$ be a proper and asymptotically normal AS.R. on $X$. We call the Smirnov compactification $(\mathfrak{X},\delta_{\lambda}^{*})$ of $X$, the \emph{asymptotic compactification} of $X$.
\end{definition}
\begin{proposition}
Let $X$ be a normal topological space and let $\mathcal{E}$ be a proper coarse structure on $X$. Assume that the AS.R. associated to $\mathcal{E}$ is asymptotically normal. Then the asymptotic compactification of $X$ is homeomorphic to $hX$.
\end{proposition}
\begin{proof}
See \cite{maa}.
\end{proof}
\begin{corollary}\label{11setare}(\cite{maa} \emph{Proposition 4.24})
 Let $X$ be a normal topological space and let $\mathcal{E}$ be a proper coarse structure on $X$. Assume that the AS.R. associated to $\mathcal{E}$ is asymptotically normal. Two subsets $A$ and $B$ of $X$ are asymptotically disjoint if and only if $\nu A\bigcap \nu B=\emptyset$.
\end{corollary}

\subsection{Asymptotic inductive dimension}\label{111}
Let us recall the definition of a separator in a topological space.
\begin{definition}
Let $X$ be a topological space. A subset $C\subseteq X$ is called a separator between disjoint subsets $A$ and $B$ of $X$, if $X\setminus C=U\bigcup V$ where $U$ and $V$ are disjoint and open and they contain $A$ and $B$ respectively.
\end{definition}
We are recalling the following two definitions from \cite{Ind} and \cite{Uni}.
\begin{definition}
Let $X$ be a topological space and let $\mathcal{E}$ be a proper coarse structure on $X$. Suppose that $A$ and $B$ are two asymptotically disjoint subsets of $X$. We call a subset $C$ of $X$ an \emph{asymptotic separator} between $A$ and $B$, if $\nu C$ is a separator between $\nu A$ and $\nu B$ in $\nu X$.
\end{definition}
\begin{definition}
Let $(X,d)$ be a proper metric space. The \emph{asymptoic inductive dimension} of $X$, denoted by  $\operatorname{asInd}X$, is defined inductively by $\operatorname{asInd}X=-1$ if and only if $X$ is bounded and for $n\in \mathbb{N}\bigcup \{0\}$, $\operatorname{asInd}X\leq n$ if for each two asymptotically disjoint subsets $A$ and $B$ of $X$ there is an asymptotic separator $C$ between them such that $\operatorname{asInd}C\leq n-1$. We say that $\operatorname{asInd}X=n$ if $\operatorname{asInd}X\leq n$ and $\operatorname{asInd}X\leq n-1$ is not true. If $\operatorname{asInd}X\leq n$ is not true for each $n\in \mathbb{N}$, we say that the asymptotic inductive dimension of $X$ is infinite.
\end{definition}
Let us recall that the \emph{asymptotic dimension} of a metric space $X$ is less than $n\in \mathbb{N}\bigcup\{0\}$ if for each $r>0$ there exists a uniformly bounded cover of $X$ with $r$-multiplicity less than $n+1$. The $r$-multiplicity of a uniformly bounded cover $\mathcal{U}$ of $X$ is the smallest number $n\in \mathbb{N}$ such that each ball of radius $r$ in $X$ intersects at most $n$ elements of $\mathcal{U}$. The following proposition is proved in \cite{Uni}.
\begin{proposition}\label{asdim}
Let $X$ be a unbounded proper metric space with finite asymptotic dimension. Then the asymptotic inductive dimension of $X$ is equal to its asymptotic dimension.
\end{proposition}

\section{Large scale inductive dimension}

\begin{definition}\label{sep}
Let $(X,\lambda)$ be an AS.R. space and let $A$ and $B$ be two asymptotically disjoint subsets of $X$. We call a subset $C$ of $X$ a \emph{large scale separator} between $A$ and $B$ if\\
i) $C$ is asymptotically disjoint from $A$ and $B$.\\
ii) $X=X_{1}\bigcup X_{2}$ such that $X_{1}$ and $X_{2}$ are asymptotically disjoint from $A$ and $B$ respectively and if $L_{1}\lambda L_{2}$ for two unbounded subsets $L_{1}\subseteq X_{1}$ and $L_{2}\subseteq X_{2}$, then there exists a subset $L$ of $C$ such that $L\lambda L_{1}$.
\end{definition}
\begin{proposition}\label{setare}
Let $X$ be a normal topological space and let $\mathcal{E}$ be a proper coarse structure on $X$. Assume that the AS.R. associated to $\mathcal{E}$ is asymptotically normal. Then each large scale separator in $X$ is an asymptotic separator.
\end{proposition}
\begin{proof}
Let $C$ be a large scale separator between two asymptotically disjoint subsets $A$ and $B$ of $X$. Assume that $X=X_{1}\bigcup X_{2}$ such that $X_{1}$ and $X_{2}$ satisfy ii) of \ref{sep}. Let $ U=(\nu X \setminus \nu X_{1})\bigcap (\nu X\setminus \nu C)$ and $V=(\nu X \setminus \nu X_{2})\bigcap (\nu X\setminus \nu C)$. By \ref{11setare}, $\nu A\subseteq U$ and $\nu B\subseteq V$. Let $\mathcal{C}$ be a cluster in $\nu X_{1}\bigcap \nu X_{2}$ and $D\in \mathcal{C}$. Since $D=(D\bigcap X_{1})\bigcup (D\bigcap X_{2})$, $D\bigcap X_{1}\in \mathcal{C}$ or $D\bigcap X_{2}\in \mathcal{C}$. Suppose that $D\bigcap X_{1}\in \mathcal{C}$. Since both  $D\bigcap X_{1}$ and $X_{2}$ are in $\mathcal{C}$, their closure are not asymptotically disjoint. Since $\lambda_{\mathcal{E}}$ and the topology of $X$ are compatible, there are unbounded subsets $L_{1}\subseteq D\bigcap X_{1}$ and $L_{2}\subseteq X_{2}$ such that $L_{1}$ and $L_{2}$ are asymptotically alike. By ii) of \ref{sep}, there exists a subset $L$ of $C$ such that $L\lambda_{\mathcal{E}} L_{1}$. So $C$ and $D$ are not asymptotically disjoint and it shows that $C\in \mathcal{C}$. Thus $\nu X_{1}\bigcap \nu X_{2}\subseteq \nu C$. Therefore $U\bigcup V=\nu X\setminus \nu C$. Since $\nu X_{1}\bigcup \nu X_{2}=\nu X$, $U\bigcap V=\emptyset$.
\end{proof}

\begin{definition}
Let $(X,\lambda)$ be an AS.R. space. By the definition the \emph{large scale inductive dimension} of $X$ is $-1$ if and only if $X$ is bounded. We say that the large scale inductive dimension of $X$ is less that $n\in \mathbb{N}$ if for each asymptotically disjoint subsets $A$ and $B$ of $X$ there is a large scale separator $C$ between them such that the large scale inductive dimension of $(C,\lambda_{C})$ is less than $n-1$. We will denote the large scale inductive dimension of $(X,\lambda)$ by $\operatorname{lsInd}_{\lambda}X$. We say $\operatorname{lsInd}_{\lambda}X=n$ if $\operatorname{lsInd}_{\lambda}X\leq n$ and $\operatorname{lsInd}_{\lambda}X\leq n-1$ is not true. We say that the large scale inductive dimension of $(X,\lambda)$ is infinite, if $\operatorname{lsInd}X\leq n$ is not true for any $n\in \mathbb{N}$.
\end{definition}
The proposition \ref{setare}, shows that for a proper metric space $(X,d)$ we have $\operatorname{asInd}X\leq \operatorname{lsInd}_{\lambda_{d}}X$.
\begin{proposition}
Let $(X,d)$ be a proper metric space. Then $\operatorname{asInd}X=0$ if and only if $\operatorname{lsInd}_{\lambda_{d}}X=0$.
\end{proposition}
\begin{proof}
If $\operatorname{lsInd}_{\lambda_{d}}X=0$, then $X$ is unbounded. Since $\operatorname{asInd}X\leq \operatorname{lsInd}_{\lambda_{d}}X$, $\operatorname{asInd}X=0$. To prove the converse, suppose that $\operatorname{asInd}X=0$ and $A$ and $B$ are two asymptotically disjoint subsets of $X$. Let $\nu X=U\bigcup V$ such that $U$ and $V$ are two open and disjoint subsets of $\nu X$ containing $\nu A$ and $\nu B$ respectively. So $U$ and $V$ are closed in $hX$. There exists a continuous function $f:hX\rightarrow [0,1]$ such that $f(U)=0$ and $f(V)=1$. Let $X_{1}=f^{-1}([\frac{1}{2},1])\bigcap X$ and $X_{2}=f^{-1}([0,\frac{1}{2}])\bigcap X$. Then $X_{1}$ and $X_{2}$ are two asymptotically disjoint subsets of $X$ such that $X=X_{1}\bigcup X_{2}$. Also $X_{1}$ and $X_{2}$ are asymptotically disjoint from $A$ and $B$ respectively.
\end{proof}
\begin{lemma}\label{dis}
Let $(X,\lambda)$ and $(Y,\lambda^{\prime})$ be two AS.R. spaces and let $f:X\rightarrow Y$ be an AS.R. mapping.\\
 i) Assume that $f$ is an asymptotic equivalence. If $A,B\subseteq X$ are two asymptotically disjoint subsets of $X$ then $f(A)$ and $f(B)$ are asymptotically disjoint subsets of $Y$.\\
 ii) If $C$ and $D$ are asymptotically disjoint subsets of $Y$, then $f^{-1}(C)$ and $f^{-1}(D)$ are two asymptotically disjoint subsets of $X$.
\end{lemma}
\begin{proof}
Notice that for an unbounded subset $L$ of $X$ since $L\subseteq f^{-1}(f(L))$, the property ii) of \ref{map} shows that $f(L)$ is unbounded.\\
i) Let $g:Y\rightarrow X$ be an AS.R. mapping such that $g\circ f(C)\lambda C$ and $f\circ g(D)\lambda^{\prime}D$ for all $C\subseteq X$ and $D\subseteq Y$. Assume that, contrary to our claim, there are unbounded subsets $L_{1}\subseteq f(A)$ and $L_{2}\subseteq f(B)$ such that $L_{1}\lambda^{\prime}L_{2}$. So $g(L_{1})\lambda g(L_{2})$. Since $g(L_{1})\subseteq g(f(A))$ and $g(L_{2})\subseteq g(f(B))$, $g(f(A))$ and $g(f(B))$ are not asymptotically disjoint. Thus by proposition \ref{6setare}, $A$ and $B$ are not asymptotically disjoint.\\
ii) The proof is straightforward.
\end{proof}
\begin{theorem}
Assume that $(X,\lambda)$ and $(Y,\lambda^{\prime})$ are two asymptotically equivalent AS.R. spaces. Then $\operatorname{lsInd}_{\lambda}X=\operatorname{lsInd}_{\lambda^{\prime}}Y$.
\end{theorem}
\begin{proof}
Let $f:X\rightarrow Y$ and $g:Y\rightarrow X$ be two AS.R. mappings such that $g\circ f(A)\lambda A$ for all $A\subseteq X$ and $f\circ g(B)\lambda^{\prime}B$ for all $B\subseteq Y$. We proceed by induction on $\operatorname{lsInd}_{\lambda^{\prime}}Y$. If $\operatorname{lsInd}_{\lambda^{\prime}}Y=-1$ then $Y$ is bounded and since $f$ is an AS.R. mapping $X=f^{-1}(Y)$ is bounded too. So $\operatorname{lsInd}_{\lambda}X=-1$. Assume that $\operatorname{lsInd}_{\lambda}X=\operatorname{lsInd}_{\lambda^{\prime}}Y$ for $\operatorname{lsInd}_{\lambda^{\prime}}Y=-1,...,n-1$. Let $\operatorname{lsInd}_{\lambda^{\prime}}Y=n$. Suppose that $A$ and $B$ are two asymptotically disjoint subsets of $X$. By \ref{dis}, $f(A)$ and $f(B)$ are asymptotically disjoint. Thus there is a large scale separator $C\subseteq Y$ between $f(A)$ and $f(B)$ such that $\operatorname{lsInd}_{\lambda^{\prime}_{C}}C\leq n-1$. There are $Y_{1},Y_{2}\subseteq Y$ such that $Y=Y_{1}\bigcup Y_{2}$ and they are asymptotically disjoint from $f(A)$ and $f(B)$ respectively. Also, if for $M_{1}\subseteq Y_{1}$ and $M_{2}\subseteq Y_{2}$ we have $M_{1}\lambda^{\prime}M_{2}$, then there exists some $M\subseteq C$ such that $M\lambda^{\prime}M_{1}$. Let $X_{1}=f^{-1}(Y_{1})$ and $X_{2}=f^{-1}(Y_{2})$. By \ref{dis}, $X_{1}$ and $X_{2}$ are asymptotically disjoint from $A$ and $B$ respectively. Let $D=g(C)$. Since $f(D)=f(g(C))\lambda^{\prime} C$ and $D\subseteq f^{-1}(f(D))$, \ref{dis} shows that $D$ is asymptotically disjoint from both $A$ and $B$. Assume that $L_{1}\subseteq X_{1}$ and $L_{2}\subseteq X_{2}$ and $L_{1}\lambda L_{2}$. Since $f(L_{1})\lambda f(L_{2})$, there exists a subset $L$ of $C$ such that $L\lambda^{\prime} f(L_{1})$. We have $g(L)\lambda g(f(L_{1}))$ and $g(f(L_{1}))\lambda L_{1}$, so $g(L)\lambda L_{1}$. By \ref{sub}, $D$ and $C$ are asymptotic equivalent, thus $\operatorname{lsInd}_{\lambda_{D}}D\leq n-1$.
\end{proof}

\section{Comparing the asymptotic dimensiongrad and the large scale inductive dimension of metric spaces}

Let $(X,d)$ be a metric space. For a subset $A$ of $X$ and $r>0$, by $\textbf{B}(A,r)$ we mean the open ball of radius $r$ around $A$.
\begin{definition}
Let $X$ be a metric space. For two subsets $A$ and $B$ of $X$, an $r$-\emph{chain} joining $A$ and $B$ is a finite sequence $x_{0},...,x_{n}$ such that $x_{0}\in A$ and $x_{n}\in B$ and $d(x_{i-1},x_{i})<r$ for $i=1,..,n$. Let $A$ and $B$ be two asymptotically disjoint subsets of $X$, we say $C$ is an \emph{asymptotic cut} between $A$ and $B$ if\\
i) It is asymptotically disjoint from both $A$ and $B$.\\
ii) For each $r>0$ there exists a positive real number $s$ such that each $r$-chain joining $A$ and $B$ meets $\textbf{B}(C,s)$.
\end{definition}
If we in the definition of large scale inductive dimension, we replace the large scale separator with the asymptotic cut, we have the definition of \emph{asymptotic dimensiongrad} of a metric space. For a metric space $X$ we denote the asymptotic dimensiongrad of $X$ by $\operatorname{asDg}X$. It can be proved that the asymptotic dimensiongrad is invariant under coarse equivalences.
\begin{proposition}\label{cts1}
Let $(X,d)$ be a metric space. Then each large scale separator in $(X,\lambda_{d})$ is an asymptotic cut.
\end{proposition}
\begin{proof}
Let $C$ be a large scale separator between asymptotic disjoint subsets $A$ and $B$ of $X$. Let $X=X_{1}\bigcup X_{2}$ such that $X_{1}$ and $X_{2}$ satisfy property ii) of \ref{sep}. Without loss of generality one can assume that $A\subseteq X_{1}$ and $B\subseteq X_{2}$. Let $r$ be a positive real number. Let $L_{1}$ be the set of all $x\in X_{1}$ such that there exists an $r$-chain $x_{0},...,x_{n}$ joining $A$ and $B$ and $x=x_{i}$ for some $i\in \{0,...,n\}$ and $x_{i+1}\in X_{2}$. Similarly, let $L_{2}$ be the set of all $y\in X_{2}$ such that there exists an $r$-chain $y_{0},...,y_{m}$ joining $A$ and $B$ and $y=y_{i}$ for some $i\in \{1,...,n\}$ and $y_{i-1}\in X_{1}$. We have $d_{H}(L_{1},L_{2})\leq r$. By ii) of \ref{sep} $d_{H}(L_{1},L)\leq s$ for some $L\subseteq C$ and some $s>0$. Therefore each $r$-chain joining $A$ and $B$ meets $\textbf{B}(C,s)$.
\end{proof}
\begin{proposition}\label{cts2}
Let $(X,d)$ be a proper metric space. Then each asymptotic separator is an asymptotic cut.
\end{proposition}
\begin{proof}
See \cite{asdim} Proposition 27.
\end{proof}
Let $(X,d)$ be a metric space. Clearly \ref{cts1} and \ref{cts2} shows that $\operatorname{asDg}X\leq \operatorname{lsInd}_{\lambda_{d}}X$, and $\operatorname{asDg}X\leq \operatorname{asInd}X$ if $X$ is proper. The inverses of \ref{cts1} and \ref{cts2} are not true.
\begin{example}
Let $X=\bigcup_{n\in \mathbb{N}}(2^{n}\times [0,+\infty))$ and let $d$ be the induced metric from $\mathbb{R}^{2}$ on $X$. Suppose that $A=\{(2^{n},0)\mid n\in \mathbb{N}\}$ and $B=2\times [0,+\infty)$. It is straightforward to show that $A$ and $B$ are asymptotically disjoint and $C=\{(2,0)\}$ is an asymptotic cut between them which is neither a large scale separator nor an asymptotic separator.
\end{example}
Let $(X,d)$ be a metric space and let $x\in X$. For a positive real number $r$, we denote by $[x]_{r}$ the set of all $y\in X$ such that there is an $r$-chain joining $\{x\}$ and $\{y\}$.
\begin{definition}
For a positive number $r$, we say a metric space $X$ is $r$-\emph{connected} if $X=[x]_{r}$ for some $x\in X$. Equivalently we can call a metric space $X$, $r$-connected if each two points in $X$ are connected by an $r$-chain.
\end{definition}
\begin{definition}
Let $X$ be an $r$-connected metric space for some $r>0$. For each $x,y\in X$, let $d_{r}(x,y)$ be the minimum $n\in \mathbb{N}\bigcup \{0\}$ such that there exists an $r$-chain $x=x_{0},...,x_{n}=y$ in $X$, joining $\{x\}$ and $\{y\}$. We call $d_{r}$ the $r$-\emph{metric}.\\
\end{definition}
\begin{definition}
We call an $r$-connected metric space $(X,d)$, $r$-\emph{convex} if $x,y\in X$ and $d(x,y)\geq r$, imply $d_{r}(x,y)\leq d(x,y)$.
\end{definition}
Note that in an $r$-connected metric space $(X,d)$ if $d(x,y)\leq r$ then $d_{r}(x,y)\leq 1$.
\begin{example}
Each geodesic metric space is $2$-convex. Every finitely generated group with the word metric is $2$-convex.
\end{example}
\begin{proposition}\label{cts3}
Let $(X,d)$ be an $r$-convex metric space for some $r>0$. Then each asymptotic cut in $X$ is a large scale separator in $(X,\lambda_{d})$.
\end{proposition}
\begin{proof}
Let $A$ and $B$ be two asymptotic disjoint subsets of $X$, and let $C$ be an asymptotic cut between them. Suppose that $A$ is bounded. Let $X_{1}=A\bigcup C$ and let $X_{2}=X\setminus X_{1}$. Since $C$ is asymptotically disjoint from $B$ it is easy to check that $X_{1}$ and $X_{2}$ satisfy the property ii) of \ref{sep}. Similar argument holds whenever $B$ is bounded. Now, suppose that $A$ and $B$ are unbounded. There is an $s>0$ such that each $r$-chain joining $A$ and $B$ meets $\textbf{B}(C,s)$. Let $Y=X\setminus \textbf{B}(C,s)$. Since $C$ is asymptotically disjoint form both $A$ and $B$, $\textbf{B}(C,s)$ contains only bounded subsets of $A\bigcup B$. So without loss of generality, we can assume that $A\bigcup B\subseteq Y$. By $[x]_{r}^{Y}$ we mean the set of all points $y\in Y$ such that there exists an $r$-chain joining $\{x\}$ and $\{y\}$ in $Y$. Let $X_{1}=\bigcup_{a\in A}[a]_{r}^{Y}$ and let $X_{2}=X\setminus X_{1}$. Suppose that $L_{1}\subseteq X_{1}$ and $L_{2}\subseteq X_{2}$ are two unbounded and asymptotically alike subsets of $X$. So $d_{H}(L_{1},L_{2})<m$, for some $m\in \mathbb{N}$. Let $x\in L_{1}$. There exists some $y\in L_{2}$ such that $d(x,y)<m$. So $d_{r}(x,y)\leq m$, and it shows that there is an $r$-chain $x=x_{0},...,x_{n}=y$ joining $\{x\}$ and $\{y\}$ in $X$ with $n\leq m$. Since $y\notin X_{1}$, there is some $i\in \{0,...,n\}$ such that $x_{i}\in \textbf{B}(C,s)$. Thus $d(x,C)\leq \sum_{j=0}^{i-1}d(x_{j},x_{j+1})+d(x_{i},C)\leq ir+s\leq mr+s $. It shows that $L_{1}\subseteq \textbf{B}(C,mr+s)$. Let $L=\textbf{B}(L_{1},mr+s)\bigcap C$. Then $L$ is a subset of $C$ which is asymptotically alike to $L_{1}$. Since each $r$-chain joining $A$ and $B$ meets $\textbf{B}(C,s)$, $A\subseteq X_{1}$ and $B\subseteq X_{2}$. Therefore $C$ is a large scale separator between $A$ and $B$.
\end{proof}
\begin{theorem}\label{asl}
 Assume that $(X,d)$ is an $r$-convex metric space, for some $r>0$. Then $\operatorname{asDg}X=\operatorname{lsInd_{\lambda_{d}}}X$, and $\operatorname{asDg}X=\operatorname{asInd}X=\operatorname{lsInd_{\lambda_{d}}}X$, if $X$ is proper.
\end{theorem}
\begin{proof}
It is a straightforward consequence of propositions \ref{setare}, \ref{cts1}, \ref{cts2} and \ref{cts3}.
\end{proof}
\begin{corollary}
Let $X$ be a proper metric space. If $X$ is coarsely equivalent to a geodesic metric space, then $\operatorname{asDg}X=\operatorname{asInd}X=\operatorname{lsInd_{\lambda_{d}}}X$.
\end{corollary}
\begin{proof}
Since $\operatorname{asDg}X$, $\operatorname{asInd}X$ and $\operatorname{lsInd_{\lambda_{d}}}X$ are invariant under coarse equivalences and since each geodesic metric space is $2$-convex, this corollary is a straightforward result of \ref{asl}.
\end{proof}
\begin{lemma}\label{hich}
Let $(X,d)$ be a metric space. If $\operatorname{asDg}X=0$, then $[x]_{r}$ is bounded for all $x\in X$ and $r>0$.
\end{lemma}
\begin{proof}
Suppose that, contrary to our claim, there are $x\in X$ and $r>0$ such that $[x]_{r}$ is unbounded. Choose $a_{1}\in [x]_{r}$ such that $d(x,a_{1})>r+3$. Let $x=x_{0},...,x_{m}=a_{1}$ be an $r$-chain joining $\{x\}$ and $\{a_{1}\}$. Since $d(\textbf{B}(x,1),\textbf{B}(a_{1},1))>r$, there exists an $i\in \{1,..,m-1\}$ such that $x_{i}$ is not in $\textbf{B}(x,1)\bigcup \textbf{B}(a_{1},1)$. Let $j=\max \{i\mid x_{i}\notin \textbf{B}(x,1)\bigcup \textbf{B}(a_{1},1)\}$. Suppose that $b_{1}=x_{j}$. So $d(b_{1},a_{1})>1$ and the $r$-chain $b=x_{j},...,x_{m}=a_{1}$ does not intersect $\textbf{B}(x,1)$. Assume that $a_{n}$ and $b_{n}$ are chosen such that $d(a_{n},b_{i})\geq n$ and $d(a_{i},b_{n})\geq n$ for all $i\in \{1,..,n\}$ and there exists an $r$-chain from $a_{n}$ to $b_{n}$ such that it does not intersect $\textbf{B}(x,n)$. Let $x=a_{0}$. Choose $a_{n+1}\in [x]_{r}$ such that $d(a_{n+1},a_{i})>2n+r+3$ for all $i\in \{0,..,n\}$ and $d(a_{n+1},b_{i})>n+1$ for all $i\in \{1,..,n\}$. Let $x=y_{0},...,y_{l}=a_{n+1}$ be an $r$-chain joining $\{x\}$ and $\{a_{n+1}\}$. Since $d(\textbf{B}(x,n+1),\textbf{B}(a_{n+1},n+1))>r$ there is an $i\in \{1,..,l\}$ such that $x_{i} \notin \textbf{B}(x,n+1)\bigcup \textbf{B}(a_{n+1},n+1)$. Let $k=\max \{i\mid y_{i} \notin \textbf{B}(x,n+1)\bigcup \textbf{B}(a_{n+1},n+1)\}$ and $b_{n+1}=y_{k}$. So the $r$-chain $b_{n+1}=y_{k},...,y_{l}=a_{n+1}$ does not intersect $\textbf{B}(x.n+1)$. In addition, we have $d(b_{n+1},\textbf{B}(a_{n+1},n+1))<r$. Since $d(a_{n+1},a_{i})>2n+r+3$ for all $i\in\{1,..,n\}$ , $d(b_{n+1},a_{i})>n+1$ for all $i\in \{1,..,n+1\}$. Let $A=\{a_{n}\mid n\in \mathbb{N}\}$ and $B=\{b_{n}\mid n\in \mathbb{N}\}$. Clearly $A$ and $B$ are asymptotically disjoint. For each $n\in \mathbb{N}$ there is an $r$-chain joining $A$ and $B$ that does not intersect $\textbf{B}(x,n)$. Therefore no bounded subset of $X$ can be an asymptotic cut between $A$ and $B$. Thus $\operatorname{asDg}X>0$.
\end{proof}
The inverse of \ref{hich} is not true in general.
\begin{example}
Let $X=\bigcup_{n\in \mathbb{N}}2^{n}\times [-n,n]$ and let $d$ be the induced metric from $\mathbb{R}^{2}$ on $X$. One can easily show that $\operatorname{asDg}X=1$, although $[x]_{r}$ is bounded for each $x\in X$ and $r>0$.
\end{example}
In \cite{count}, Smith showed that each countable group has a proper left invariant metric which is unique up to coarse equivalence. It can be shown that a countable group $G$ has asymptotic dimension zero if and only if each finitely generated subgroup of $G$ is finite (\cite{count}). We can prove the inverse of \ref{hich} for all countably infinite groups.
\begin{proposition}
Let $G$ be a countably infinite group with a proper and left invariant metric $d$. Then $\operatorname{asDg}G=0$ if and only if for each $x\in X$ and $r>0$, $[x]_{r}$ is bounded.
\end{proposition}
\begin{proof}
The "only if" part is a consequence of \ref{hich}. To prove the converse, suppose that $F$ is a finite subset of $G$. Let $r=\max \{d(e,g)\mid g\in F\}$, where $e$ is the neutral element of $G$. Assume that $H=<F>$ is the subgroup of $G$ generated by $F$. Clearly $H\subseteq [e]_{r}$. So each finitely generated subgroup of $G$ is finite. Therefore the asymptotic dimension of $G$ is zero. By \ref{asdim}, $\operatorname{asInd}G=0$. Since $\operatorname{asDg}G\leq \operatorname{asInd}X$ and $G$ is unbounded, $\operatorname{asDg}G=0$.
\end{proof}
\begin{theorem}\label{zero}
Let $(X,d)$ be a metric space. Then $\operatorname{asDg}X=0$ if and only if $\operatorname{lsInd_{\lambda_{d}}}X=0$.
\end{theorem}
\begin{proof}
If  $\operatorname{lsInd_{\lambda_{d}}}X=0$, then $X$ is unbounded. Since $\operatorname{asDg}X\leq \operatorname{lsInd_{\lambda_{d}}}X$, $\operatorname{asDg}X=0$. To prove the converse, assume that $\operatorname{asDg}X=0$. Let $A$ and $B$ be two asymptotically disjoint subsets of $X$. Let $A_{i}=\bigcup_{a\in A}[a]_{i}\setminus \bigcup_{b\in B}[b]_{i}$ for $i\in \mathbb{N}$. Assume that $A_{0}=A$ and let $X_{1}=\bigcup_{i=0}^{\infty}A_{i}$. Let $X_{2}=X\setminus X_{1}$. Clearly $B\subseteq X_{2}$. We claim that $X_{1}$ and $X_{2}$ are asymptotically disjoint. On the contrary to our claim, suppose that there exist unbounded subsets $L_{1}\subseteq X_{1}$ and $L_{2}\subseteq X_{2}$ such that $d_{H}(L_{1},L_{2})<m$ for some $m\in \mathbb{N}$. Let $x\in L_{1}$. There exists some $y\in L_{2}$ such that $d(x,y)<m$. For $i\geq m$ and $d\in X$, $x\in [d]_{i}$ if and only if $y\in [d]_{i}$. Using this fact, it is straightforward to show that $x\notin A_{i}$ for all $i\geq m$. Thus $L_{1}\bigcap A_{i}=\emptyset$ for all $i\geq m$. So $L_{1}\subseteq \bigcup_{i=0}^{m-1}A_{i}$. Since $L_{1}$ is unbounded, there is some $j\in \{0,...,m-1\}$ such that $L_{1}\bigcap A_{j}$ is unbounded. Let $L_{1}^{\prime}=\{a\in A\mid [a]_{j}\bigcap L_{1}\neq \emptyset\}$. Suppose that $L_{1}^{\prime}\subseteq \textbf{B}(x,k)$ for some $k>0$ and $x\in X$. For $s=\max\{k,r\}$, $L_{1}\bigcap A_{j}\subseteq [x]_{s}$. Since $L_{1}\bigcap A_{j}$ is unbounded, $[x]_{s}$ is unbounded and it contradicts \ref{hich}. Thus $L_{1}^{\prime}$ is unbounded. Let $a\in L_{1}^{\prime}$. Choose $x\in L_{1}\bigcap [a]_{j}$. There is a $y\in L_{2}$ such that $d(x,y)<m$. Since $j< m$, $y\in [a]_{m}$. Since $y\notin A_{m}$, there is a $b\in B$ such that $y\in [b]_{m}$. Thus for each $a\in L_{1}^{\prime}$ there is a $b\in B$ such that there is an $m$-chain joining them. Since $L_{1}^{\prime}$ and $B$ are asymptotically disjoint and $\operatorname{asDg}X=0$, there is a point $z\in X$ such that it is an asymptotic cut between $L_{1}^{\prime}$ and $B$. So there exists an $n\in \mathbb{N}$ such that each $m$-chain joining $L_{1}^{\prime}$ and $B$ intersects $\textbf{B}(z,m)$. Therefore $[z]_{l}$ is unbounded for $l=\max\{n,m,j\}$, and it contradicts \ref{hich}.
\end{proof}

\end{document}